\documentclass[10pt,a4paper]{article}

\usepackage[hyperfootnotes=false,pdfpage layout=SinglePage]{hyperref}
\usepackage{mathptmx}
\usepackage[scaled=.92]{helvet}
\usepackage{amssymb}
\usepackage{amsmath}
\usepackage{amsthm}
\usepackage{latexsym}
\usepackage{amsfonts}
\usepackage{mathrsfs}
\usepackage{setspace}

\newtheorem{theorem}{Theorem}[section]
\newtheorem{corollary}[theorem]{Corollary}
\newtheorem{lemma}[theorem]{Lemma}

\theoremstyle{definition}
\newtheorem{definition}[theorem]{Definition}
\newtheorem{remark}[theorem]{Remark}

\numberwithin{equation}{section}
\renewcommand{\emptyset}{\varnothing}

\newcommand{\R}{\ensuremath{\mathbb R}}    
\newcommand{\C}{\ensuremath{\mathbb C}}    
\newcommand{\N}{\ensuremath{\mathbb N}}    

\newcommand{\gperp}{{[\perp]}}

\newcommand{\product}{[\cdot\,,\cdot]}
\newcommand{\hproduct}{(\cdot\,,\cdot)}

\newcommand{\calA}{\mathcal A}

\newcommand{\calH}{\mathcal H}

\newcommand{\calL}{\mathcal L}         
\newcommand{\calM}{\mathcal M}

\newcommand{\calU}{\mathcal U}

\newcommand{\la}{\lambda}
\newcommand{\veps}{\varepsilon}

\renewcommand{\Im}{\operatorname{Im}}
\renewcommand{\Re}{\operatorname{Re}}

\renewcommand{\ker}{\operatorname{ker}}
\newcommand{\ran}{\operatorname{ran}}

\newcommand{\sap}{\sigma_{{ap}}}

\newcommand{\spp}{\sigma_+}
\newcommand{\smm}{\sigma_-}

\newcommand{\ol}{\overline}

\newcommand{\wt}{\widetilde}

\newcommand{\Skdef}{(\raisebox{0.5 ex}{.},\raisebox{0.5 ex}{.})}
\newcommand{\Skindef}{[\raisebox{0.5 ex}{.},\raisebox{0.5 ex}{.}]}

\begin{document}
\vspace*{-.3cm}
\begin{center}
\begin{spacing}{1.7}
{\LARGE\bf Local spectral theory for normal operators in Krein spaces}
\end{spacing}

\vspace{1cm}
{\Large Friedrich Philipp, Vladimir Strauss and Carsten Trunk}
\end{center}

\vspace{1cm}\hrule\vspace*{.4cm}
\noindent{\bf Abstract}

\vspace*{.3cm}\noindent
Sign type spectra are an important tool in the investigation of spectral
properties of selfadjoint operators in Krein spaces.  It is our aim to show
that also sign type spectra for normal operators in Krein spaces
provide insight in the spectral nature of the operator:
If the real part and the
imaginary part of a normal operator in a Krein space have real spectra
only and if the growth of the resolvent of the imaginary part (close to
the real axis)  is of finite order, then the normal operator possesses
a  local spectral function defined for Borel subsets of the spectrum
which belong to positive
(negative) type spectrum. Moreover, the restriction of the normal
operator to the spectral subspace corresponding to such a Borel subset
is a normal operator in some Hilbert space. In particular, if the
spectrum consists entirely out of positive and negative type spectrum,
then the operator  is similar
to a normal operator in some Hilbert space. We use this result to
show the existence of operator roots of a class of quadratic operator
polynomials with normal coefficients.
\vspace*{.4cm}\hrule
%
%

\vspace*{.5cm}
\section*{Introduction}
Recall that a bounded operator $N$ in a Krein space  $(\calH,\product)$ is normal if $NN^+ = N^+N$, where $N^+$ denotes the adjoint operator of $N$ with respect to the Krein space (indefinite) inner product $\product$. In contrast to (definitizable) selfadjoint operators in Krein spaces, the knowledge about normal operators is very restricted.

Some results exist for normal operators in Pontryagin spaces. The starting point  is a result of M.A.\ Naimark, see \cite{N}, which implies that for a normal operator N in a  Pontryagin space $\Pi_\kappa$ there exists a $\kappa$-dimensional non-positive common invariant subspace for $N$ and its adjoint $N^+$. In \cite{LS06,CH} spectral properties of normal operators in Pontryagin spaces were considered and, in the case $\Pi_1$,  a classification of the normal operators is given.

There is only a very limited number of results in the study of normal operators in spaces others than Pontryagin spaces. In \cite{ES} a definition of definitizable normal operators was given and it was proved that a bounded normal definitizable operator in a Banach space with a regular Hermitian form has a spectral function with finitely many critical points. Let us note that in this case the spectral function is a homomorphism from the Borel sets containing no critical points on their boundaries to a commutative algebra of normal projections, see also \cite{AJ}. Some advances for Krein spaces without the assumption of definitizability can be found in \cite{AS}. We mention that \cite{AJ} contains some perturbation results for fundamentally reducible normal operators. The case of fundamentally reducible and strongly stable normal operators is considered in \cite{B98,B99}.

On the other hand, the spectral theory for definitizable (and locally definitizable) selfadjoint operators in Krein spaces is well-developed (see, e.g., \cite{L82,J,AJT} and references therein). One of the main features of definitizable selfadjoint operators in Krein spaces is their property to act locally (with the exception of at most finitely many points) similarly as a selfadjoint operator in some Hilbert space. More precisely, the spectrum of a definitizable operator consists of  spectral points of positive and of negative type, and of finitely many exceptional (i.e.\ non-real or critical) points, see \cite{L65}. For a real point $\lambda$ of positive (negative) type of a selfadjoint operator in a Krein space there exists a local spectral function $E$ such that $(E(\delta){\mathcal H}, \product)$ {\rm (}resp.\ $(E(\delta){\mathcal H}, -\product)${\rm )} is a Hilbert space for (small) neighbourhoods $\delta$ of $\lambda$.

In \cite{LaMM,LMM} a characterization for  spectral points of positive (negative) type was given in terms of normed approximate eigensequences. If all accumulation points of the sequence $([x_{n},x_{n}])$ for each normed approximate eigensequences corresponding to $\lambda$ are positive (resp.\ negative) then $\lambda$  is a spectral point of positive (resp.\ negative) type.
Obviously, the above characterization can be used as a definition for spectral points of positive (negative) type for arbitrary (not necessarily selfadjoint) operators in Krein spaces (as it was done in \cite{ABJT}). It is the main result of this paper that also for a normal operator $N$ in a Krein space $(\mathcal H,\product)$ positive and negative type spectrum implies the existence of a local spectral function for $N$. However, for this we have to impose some additional assumptions: The spectra of the real and imaginary part of $N$ are real  and the growth of the resolvent of the imaginary part (close to the real axis) of $N$  is of finite order. Under these assumptions we are able to show that $N$ has a local spectral function $E$ on each closed rectangle which consists only of spectral points of positive type or of points from the resolvent set of $N$. The local spectral function $E$ is then defined for all Borel subsets $\delta$ of this rectangle and
$E(\delta)$ is a selfadjoint projection in the Krein space  $(\calH,\product)$.
It has the property that $(E(\delta)\mathcal H,\product)$ is a Hilbert spaces for all such $\delta$. This implies that the restriction of $N$ to the spectral subspace $E(\delta)\mathcal H$ is a normal operator in the Hilbert space $(E(\delta)\mathcal H,\product)$.

We emphasize that this result provides a simple sufficient condition for the normal operator $N$ to be similar to a normal operator in a Hilbert space: If each spectral point of $N$ is of positive or of negative type and if the spectra of the real and imaginary part of $N$ are real  and the growth of the resolvent of the imaginary part is of finite order, then  $N$ is similar to a normal operator in a Hilbert space. Actually, in the final section, we use this result
to prove the existence of an operator root of a quadratic operator pencil with normal coefficients.

\section{Some auxiliary statements}
In this section we collect some statements on bounded operators in Banach spaces. As usual, by $L(X,Y)$ we denote the set of all bounded linear operators acting between Banach spaces $X$ and $Y$ and set $L(X) := L(X,X)$. In this paper a subspace is always a closed linear manifold. The approximate point spectrum $\sap(T)$ of a bounded linear operator $T$ in a Banach space $X$ is the set of all $\la\in\C$ for which there exists a sequence $(x_n)\subset X$ with $\|x_n\|=1$ for all $n\in\N$ and $(T - \la)x_n\to 0$ as $n\to\infty$. A point in $\sap(T)$ is called an {\it approximate eigenvalue} of $T$. We have
\begin{equation}\label{e:incs}
\partial\sigma(T)\,\subset\,\sap(T)\,\subset\,\sigma(T),
\end{equation}
see \cite[Chapter VII, Proposition 6.7]{C}. Therefore, $\sap(T)\neq\emptyset$ if $X\neq\{0\}$.

The following Lemmas \ref{l:polynom}--\ref{l:rosenblum} are well-known. For their proofs we refer to Lemma 0.11, Theorem 1.3, Theorem 0.8 and Corollary 0.13 in \cite{RR}.

\begin{lemma}\label{l:polynom}
Let $S$ and $T$ be two commuting bounded operators in a Banach space $X$ and let $p$ be a polynomial in two variables. Then
$$
\sigma(p(S,T))\,\subset\,\{p(\la,\mu) : \la\in\sigma(S),\,\mu\in\sigma(T)\}.
$$
If, in addition, the operators $S + T$ and $i(S - T)$ have real spectra, i.e.
\begin{equation}\label{realspectra}
\sigma(S + T)\subset\R\quad\text{ and }\quad\sigma(S - T)\subset i\R
\end{equation}
then the following identity holds:
$$
\sigma(p(S,T)) = \{p(\la,\ol\la) : \la\in\sigma(S)\}.
$$
In particular, we have
\begin{align*}
\sigma\left(\frac{S+T}{2}\right) &= \{\Re\la : \la\in\sigma(S)\},\\
\sigma\left(\frac{S-T}{2i}\right) &= \{\Im\la : \la\in\sigma(S)\}.
\end{align*}
\end{lemma}

\begin{lemma}\label{l:invariant_spec}
Let $T$ be a bounded operator in a Banach space $X$ and let $\calL$ be a subspace of $X$ which is invariant with respect to $T$. Then
$$
\sigma(T|\calL)\subset\sigma(T)\cup\rho_b(T),
$$
where $\rho_b(T)$ is the union of all bounded connected components of $\rho(T)$. In particular, if $\sigma(T)\subset\R$, we have
$$
\sigma(T|\calL)\subset\sigma(T).
$$
\end{lemma}

\begin{lemma}[Rosenblum's Corollary]\label{l:rosenblum}
Let $S$ and $T$ be bounded operators in the Banach spaces $\mathcal X$ and $\mathcal Y$, respectively. If $\sigma(S)\cap\sigma(T) = \emptyset$, then for every $Z\in L(\mathcal Y,\mathcal X)$ the operator equation
$$
SX - XT = Z
$$
has a unique solution $X\in L(\mathcal Y,\mathcal X)$. In particular, $SX = XT$ implies $X = 0$.
\end{lemma}

\begin{remark}
We mention that a much more general version of Lemma \ref{l:rosenblum} can be found in \cite[Theorem 3.1]{DK}. In addition, an explicit formula for the solution $X$ is given in \cite{DK}. Here, we will only make use of the last assertion in Lemma \ref{l:rosenblum}.
\end{remark}

Let $T$ be a bounded operator in a Banach space and let $Q\subset\C$ be a compact set. We say that a subspace $\calL_Q$ is {\it the maximal spectral subspace of $T$ corresponding to $Q$} if $\calL_Q$ is $T$-invariant, $\sigma(T|\calL_Q)\subset\sigma(T)\cap Q$ and if $\calL\subset\calL_Q$ holds for every $T$-invariant subspace $\calL$ with $\sigma(T|\calL)\subset Q$. Recall that such a subspace is hyperinvariant with respect to $T$, i.e.\ it is invariant with respect to each bounded operator which commutes with $T$ (see \cite[Chapter 1, Proposition 3.2]{CF}).

If the spectrum of the bounded operator $T$ is real, we say that the growth of the resolvent of $T$ is of {\it finite order $n$}, $n\in\N\setminus\{0\}$, if for some  $c > 0$ there exists an $M > 0$, such that
\begin{equation}\label{e:polres}
0 < |\Im\la| < c\quad\Longrightarrow\quad\|(T - \la)^{-1}\|\;\le\;\frac{M}{|\Im\la|^n}.
\end{equation}
Since the function $\rho\mapsto M/\rho^n$, $0 < \rho < 1$, satisfies the Levinson condition (cf.\ \cite[formula (2.1.2)]{LM}), it is a consequence of \eqref{e:polres} and \cite[Chapter II, \textsection 2, Theorem 5]{LM} that to each compact interval $\Delta$ the maximal spectral subspace $\calL_\Delta$ of $T$ corresponding to $\Delta$ exists.

By $r(T)$ we denote the spectral radius of a bounded operator $T$ in a Banach space.

\begin{lemma}\label{l:power}
Let $T\neq 0$ be a bounded operator in a Banach space with real spectrum such that the growth of its resolvent is of order $n$. Then for all $k\ge n$ we have
$$
\left\|T^k\right\|\;\le\; 2^k\|T\|^{k-n}\big(M + \|T\|^{n-1}\big)\,r(T),
$$
where $M = \sup\{|\Im\la|^n\|(T - \la)^{-1}\| : 0 < |\Im\la| < \|T\|\}$.
\end{lemma}
\begin{proof}
For $\rho > 0$ we define the function
$$
M(\rho) = \sup\{|\Im\la|^n\|(T - \la)^{-1}\| : 0 < |\Im\la| < \rho\}.
$$
It is obvious that this function is non-decreasing and continuous. Therefore, the infimum $M(0):=\inf_{\rho > 0}\,M(\rho)$ exists. We have $M=M(\|T\|)$.

Let $k\ge n$. Let $\mathcal C$ be the circle with center $0$ and radius $\rho > r(T)$. For $0 < |\Im\la| < \rho$ we have
\begin{equation}\label{e:polres2}
\|(T - \la)^{-1}\|\;\le\;\frac{M(\rho)}{|\Im\la|^n}.
\end{equation}
Observe that for $j\in\N$, $j \ge 1$, the function $\lambda\mapsto\la^{-j}(T - \la)^{-1}$ is holomorphic outside of $\mathcal C$. Due to $\|(T - \la)^{-1}\| = O(|\la|^{-1})$ as $|\la|\to\infty$, the Cauchy integral theorem and standard estimates of contour integrals we obtain
$$
\int_{\mathcal C}\,\la^{-j}(T - \la)^{-1}\,d\la = 0,\quad j\ge 1.
$$
Therefore, the relation
$$
\left(\frac{\la^2 - \rho^2}{\la}\right)^k = \sum_{j=0}^k\,\binom k j\,\la^{2j-k}\left(-\rho^2\right)^{k-j}
$$
yields
$$
-\frac 1 {2\pi i}\,\int_{\mathcal C}\,\left(\frac{\la^2 -
\rho^2}{\la}\right)^k (T - \la)^{-1}\,d\la = \sum_{j = \lceil
k/2\rceil}^k\,\binom k j\,(-\rho^2)^{k-j}T^{2j-k},
$$
where $\lceil k/2\rceil$ denotes the smallest integer larger than $k/2$.
Since $k\ge n$ and $\left|\frac{\la^2 - \rho^2}{\la}\right| = 2|\Im\la|$ for $\la\in\mathcal C$, together with \eqref{e:polres2} this gives
$$
\left\|T^k + \sum_{j = \lceil k/2\rceil}^{k-1}\,\binom k j\,(-\rho^2)^{k-j}T^{2j-k}\right\| \le 2^k M(\rho)\rho^{k-n+1},
$$
and hence
$$
\left\|T^k\right\|\le \left(2^k M(\rho)\rho^{k-n} + \sum_{j = \lceil k/2\rceil}^{k-1}\,\binom k j\,\rho^{2(k-j)-1}\|T\|^{2j-k}\right)\,\rho.
$$
Letting $\rho\to r(T)$ we obtain
\begin{align*}
\left\|T^k\right\|
&\le \left(2^k M(r(T))\|T\|^{k-n} + \sum_{j=0}^{k}\,\binom k j\,\|T\|^{2(k-j)-1}\|T\|^{2j-k}\right)\,r(T).
\end{align*}
We have $M(r(T))\le M(\|T\|)$,
which leads to the desired estimate
with $M= M(\|T\|)$.
\end{proof}

For a finite interval $\Delta$ we denote by $\ell(\Delta)$  the length of $\Delta$.

\begin{corollary}\label{c:power}
Let $T$ be as in Lemma {\rm \ref{l:power}}.
Then there exists $C > 0$ such that for each $k\ge n$, each $\la\in\sigma(T)$ and each compact interval $\Delta$ with $\la\in\Delta$ and $\ell(\Delta)\le\|T\|$ we have
$$
\left\|(T|\calL_\Delta - \la)^k\right\| \;\le\; 4^k\|T\|^{k}C\cdot\ell(\Delta),
$$
where $\calL_\Delta$ denotes the maximal spectral subspace of $T$ corresponding to $\Delta$.
\end{corollary}
\begin{proof}
We have $\sigma(T_\Delta)\subset\Delta$, where $T_\Delta := T|\calL_\Delta$. Clearly, the growth of the resolvent of $T_\Delta - \la$ is of order $n$. Since $\|T_\Delta - \la\|\le\|T\| + |\la|\le 2\|T\|$ and $r(T_\Delta - \la)\le\ell(\Delta)$, Lemma \ref{l:power} gives the estimate
$$
\left\|\left(T_\Delta - \la\right)^k\right\|
\le 2^k\left(2\|T\|\right)^{k-n}\left(\wt M + 2^{n-1}\|T\|^{n-1}\right)\ell(\Delta)
$$
with $\wt M = \sup\{|\Im\mu|^n\|(T_\Delta - \la - \mu)^{-1}\| : 0 < |\Im\mu| < \|T_\Delta - \la\|\}$. As $\la$ is real,
\begin{align*}
\wt M &\le \sup\{|\Im\mu|^n\|(T - \la - \mu)^{-1}\| : 0 < |\Im\mu| < 2\|T\|\}\\
&\le \sup\{|\Im\mu|^n\|(T - \mu)^{-1}\| : 0 < |\Im\mu| < 2\|T\|\},
\end{align*}
which is independent of $\Delta$, $k$ and $\la$.
\end{proof}

\section{Spectral points of positive type of bounded operators in G-spaces}
Recall that an inner product space $(\calH,\product)$ is called a {\it Krein space} if there exist subspaces $\calH_+$ and $\calH_-$ such that
$(\calH_+,\product )$ and $(\calH_-, -\product )$ are Hilbert spaces
and
\begin{equation}\label{funddecomp}
 \calH = \calH_+\,\dotplus\,\calH_-,
\end{equation}
where $\dotplus$ denotes the direct sum of subspaces. We refer to \eqref{funddecomp} as a {\it fundamental decomposition}  of the Krein space
$(\calH,\product)$.

An inner product space $(\calH,\product)$ is called a {\it $G$-space} if $\calH$ is a Hilbert space and the inner product $\product$ is continuous with respect to the norm $\|\cdot\|$ on $\calH$, that is, there exists $c>0$ such that
$$
|[x,y]|\,\le c\|x\|\|y\|\quad\text{for all }x,y\in\calH.
$$
Let $\hproduct$ be a Hilbert space inner product on $\calH$ inducing $\|\cdot\|$. Then the inner products $\hproduct$ and $\product$ are connected via
$$
[x,y] = (Gx,y),\quad x,y\in\calH,
$$
where $G\in L(\calH)$ is a uniquely determined selfadjoint operator in $(\calH,\hproduct)$. It is well known that $(\calH,\product)$ is a Krein space if and only if $G$ is boundedly invertible, see, e.g.\ \cite{B,AI}. A bounded operator $A$ in the $G$-space $\calH$ is said to be {\it $\product$-selfadjoint} or {\it $G$-selfadjoint} if
\begin{equation}\label{Gsa}
[Ax,y] = [x,Ay]
\end{equation}
holds for all $x,y\in\calH$.

\begin{remark}
Note that in a $G$-space it is in general not possible to define a bounded adjoint with respect to $\product$ of a bounded operator. However, in a Krein space this is possible. In this case, the  usual notion of selfadjointness in a Krein space coincides with $\product$-selfadjointness in $G$-spaces.
\end{remark}

Spectral points of definite type, defined below for bounded operators in a $G$-space, were defined for $\product$-selfadjoint operators in $G$-spaces in \cite{LMM} and in \cite{ABJT} for arbitrary operators (and relations) in Krein spaces.

\begin{definition}\label{definition++}
For a  bounded operator $A$ in the $G$-space $(\calH,\product)$ a point $\la\in\sap(A)$ is called a spectral point of {\it positive} {\rm (}{\it negative}{\rm )} {\it type of $A$} if for every sequence $(x_n)$ with $\|x_n\|=1$ and $\|(A - \la)x_n\|\to 0$ as $n\to\infty$, we have
$$
\liminf_{n\to\infty}\, [x_n,x_n] > 0 \quad
\left(\,\limsup_{n\to\infty}\,[x_n,x_n] < 0, \,\text{respectively}\right).
$$
We denote the set of all points of positive (negative) type of $A$ by $\spp(A)$ ($\smm(A)$, respectively). A set $\Delta\subset\C$ is said to be {\it of positive} ({\it negative}) {\it type with respect to $A$} if every approximate eigenvalue of $A$ in $\Delta$ belongs to $\spp(A)$ ($\smm(A)$, respectively).
\end{definition}

\begin{remark}\label{r:spp_real}
If the operator $A$ is $\product$-selfadjoint, then the sets $\spp(A)$ and $\smm(A)$ are contained in $\R$ (cf.\ \cite{LMM}).
\end{remark}

The following lemma is well known for selfadjoint operators in Krein spaces and $\product$-selfadjoint operators in $G$-spaces (see e.g.\ \cite{AJT,LMM}). The proof for arbitrary bounded operators remains essentially the same. However, for the convenience of the reader we give a short proof here.

\begin{lemma}\label{l:compact}
Let $A$ be a bounded operator in the $G$-space $(\calH,\product)$. Then a compact set $K\subset\C$ is of positive type with respect to $A$ if and only if there exist a neighbourhood $\calU$ of $K$ in $\C$ and numbers $\veps,\delta > 0$ such that for all $x\in\calH$ and each $\la\in\calU$ we have
$$
\|(A - \la)x\|\le\veps\|x\|\quad\Longrightarrow\quad
[x,x]\ge\delta\|x\|^2.
$$
In this case, the set $\calU$ is of positive type with respect to $A$.
\end{lemma}
\begin{proof}
Assume that $K$ is a compact set of positive type with respect to $A$, i.e.\ $K\cap\sap(A)\subset\spp(A)$. Let $\la_0\in K$. Then it follows from Definition \ref{definition++} and the properties of the points of regular type of $A$ that there exist $\veps_0,\delta_0 > 0$ such that for all $x\in\calH$ we have
$$
\|(A - \la_0)x\|\le 2\veps_0\|x\|\quad\Longrightarrow\quad
[x,x]\ge\delta_0\|x\|^2.
$$
From this we easily conclude that for all $x\in\calH$ and all $\la\in\C$ with $|\la - \la_0| < \veps_0$ we have
$$
\|(A - \la)x\|\le\veps_0\|x\|\quad\Longrightarrow\quad
[x,x]\ge\delta_0\|x\|^2.
$$
Since $\la_0$ was an arbitrary point in $K$, the assertion follows from the compactness of $K$. The converse statement is evident.
\end{proof}

One of the main results of \cite{LMM} is that under a certain condition a $\product$-selfadjoint operator in a $G$-space has a local spectral function of positive type on intervals which are of positive type with respect to the operator. Let us recall the definition of such a local spectral function and the exact statement for $\product$-selfadjoint operators.

\begin{definition}\label{locspecfct}
Let $(\calH,\product)$ be a $G$-space, $A\in L(\calH)$ and $S\subset\C$. A set function $E$ mapping from the system $\mathcal B(S)$ of Borel-measurable subsets of $S$ whose closure is also contained in $S$ to $L(\calH)$ is called a {\it local spectral function of positive type} of the operator $A$ on $S$ if for all $Q,Q_1,Q_2,\ldots\in\mathcal B(S)$ the following conditions are satisfied:
\begin{enumerate}
\item[(i)]   $(E(Q)\calH,\product)$ is a Hilbert space and $E(Q)$ is
 $\product$-selfadjoint.
\item[(ii)]  $E(Q_1\cap Q_2) = E(Q_1)E(Q_2)$.
\item[(iii)] If $Q_1,Q_2,\ldots\in\mathcal B(S)$ are mutually disjoint, then
$$
E\left(\bigcup_{k=1}^\infty\,Q_k\right) = \sum_{k=1}^\infty\,E(Q_k),
$$
where the sum converges in the strong operator topology.
\item[(iv)]  $AB = BA\;\;\Longrightarrow\;\;E(Q)B=BE(Q)$ \;for every $B\in L(\calH)$.
\item[(v)]   $\sigma(A|E(Q)\calH)\subset\ol{\sigma(A)\cap Q}$.
\item[(vi)]  $\sigma(A|(I-E(Q))\calH)\subset\ol{\sigma(A)\setminus Q}$.
\end{enumerate}
\end{definition}

Note that (ii) implies that $E(Q)$ is a projection for all $Q\in\mathcal B(S)$ and that from (iii) (or (v)) it follows that $E(\emptyset) = 0$. By $\C^+$ ($\C^-$) we denote the open upper (lower, respectively) halfplane of the complex plane $\C$.

\begin{theorem}\label{t:lmm}
Let $A$ be a $\product$-selfadjoint operator in the $G$-space $(\calH,\product)$. If the interval $\Delta$ is of positive type with respect to $A$ and if each of the sets $\rho(A)\cap\C^+$ and $\rho(A)\cap\C^-$ accumulates to each point of $\Delta$, respectively, then $A$ has a local spectral function $E$ of positive type on $\Delta$. For each closed interval $\delta\subset\Delta$ the subspace $E(\delta)\calH$ is the maximal spectral subspace of $A$ corresponding to $\delta$.
\end{theorem}

\section{Locally definite normal operators in Krein spaces}
For the rest of this paper let $(\calH,\product)$ be a Krein space. It is our aim to extend Theorem \ref{t:lmm} to normal operators in Krein spaces. Recall that a bounded operator $N$ in a Krein space $(\calH,\product)$ is called {\it normal} if it commutes with its adjoint $N^+$, i.e.
$$
N^+N = NN^+.
$$
By definition the {\it real part} of a bounded operator $C$ in a Krein space $(\calH,\product)$ is the operator $(C+C^+)/2$ and the {\it imaginary part} is given by $(C-C^+)/2i$. It is clear that both real and imaginary part of an arbitrary bounded operator are $\product$-selfadjoint. Moreover, it is easy to see that a bounded operator in $(\calH,\product)$ is normal if and only if its real part and its imaginary part commute.

\begin{lemma}\label{l:only_sap}
Let $N$ be a normal operator in the Krein space $(\calH,\product)$.
If $\Re N$ and $\Im N$ have real spectra only, then $\sigma(N) = \sap(N)$.
\end{lemma}
\begin{proof}
Assume that $\la\in \sigma(N)\setminus\sap(N)$. Then $N - \la$ has a trivial kernel and $\ran(N - \la)\neq\calH$ is closed. Hence, $\ol\la\in\sigma_p(N^+)$. Set $\calL := \ker(N^+ - \ol\la)$. This subspace is $N$-invariant. By Lemma \ref{l:invariant_spec} the operators $\Re N|\calL$ and $\Im N|\calL$ have real spectra. Thus, by Lemma \ref{l:polynom} (with $S = N^+|\calL$ and $T = N|\calL$) we conclude that $\sigma(N|\calL) = \{\ol\mu : \mu\in\sigma(N^+|\calL)\} = \{\la\}$. Hence, $\la\in\sap(N|\calL)\subset\sap(N)$. A contradiction.
\end{proof}

The following theorem is the main result of this section.

\begin{theorem}\label{t:main}
Let $N$ be a normal operator in the Krein space $(\calH,\product)$. If $\Re N$ and $\Im N$ have real spectra and the growth of the resolvent of $\Im N$ is of finite order, then $N$ has a local spectral function of positive type on each closed rectangle $[a,b]\times [c,d]$ which is of positive type with respect to $N$.
\end{theorem}
\begin{proof}
Let $[a,b]\times [c,d]$ be of positive type with respect to $N$. Together with Lemma \ref{l:only_sap} we have
$$
([a,b]\times [c,d])\cap\sigma(N)\subset\spp(N).
$$
By Lemma \ref{l:compact} there exist an open neighbourhood $\calU$ of $[a,b]\times [c,d]$ in $\C$ and numbers $\veps,\delta\in (0,1)$ such that
\begin{equation}\label{e:epsdel}
\la\in\calU,\;\;x\in\calH,\;\;\|(N - \la)x\|\le\veps\|x\|\quad\Longrightarrow\quad [x,x]\ge\delta\|x\|^2.
\end{equation}
By Corollary \ref{c:power}, there exists a value $\tau > 0$ such that for each compact interval $\Delta$ with length $\ell(\Delta) < \tau$ and any $\lambda\in\Delta\cap\sigma(\Im N)$ we have
\begin{equation}\label{e:NormEst}
\left\|(\Im N|\calL_\Delta - \lambda)^k\right\|\,\le\,\frac{\delta^{k-1}\veps^k}{2^k} \quad\text{for
all }\,k = k_0,k_0+1,\dots,2k_0,
\end{equation}
where $k_0$ is the order of growth of the resolvent of $\Im N$ and $\calL_\Delta$ is the maximal spectral subspace of $\Im N$ corresponding to the interval $\Delta$.

The proof will be divided into three steps. In the first step we define the spectral subspace corresponding to rectangles $\Delta_1\times\Delta_2\subset\calU$ with $\ell(\Delta_2) < \tau$. In the second step we prove some properties of the spectral subspaces defined in step 1. In the third step we define the spectral subspace corresponding to the rectangle $[a,b]\times [c,d]$ and complete the proof.

{\bf 1. } Let $\Delta_1$ and $\Delta$ be compact intervals such that $\Delta_1\times\Delta\subset\calU$ and $\ell(\Delta) < \tau$. Note that the inner product space $(\calL_\Delta,\product)$ is a $G$-space which is not necessarily a Krein space. Since a maximal spectral subspace is hyperinvariant (see, e.g.\ \cite{CF}), the space $\calL_\Delta$ is invariant with respect to $N$, $N^+$, $\Re N$ and $\Im N$. By $A_0$, $B_0$, $N_0$ and $N_{0,+}$ denote the restrictions of $\Re N$, $\Im N$, $N$ and $N^+$ to $\calL_\Delta$, respectively. Then we have, see Lemma \ref{l:invariant_spec},
\begin{equation}\label{e:basic}
\sigma(A_0)\subset\sigma(\Re N)\subset\R\quad\text{and}\quad\sigma(B_0)\subset\sigma(\Im N)\cap\Delta.
\end{equation}
Moreover, from $N_0 = A_0 + iB_0$, $N_{0,+} = A_0 - iB_0$, \eqref{e:basic} and Lemma \ref{l:polynom} we conclude
$$
\sigma(A_0) = \{\Re\la : \la\in\sigma(N_0)\}
\quad\text{ and }\quad
\sigma(B_0) = \{\Im\la : \la\in\sigma(N_0)\},
$$
hence
$$
\sigma(N_0)\subset\,\sigma(A_0)\times\Delta,
$$
The operator $A_0$ is obviously $\product$-selfadjoint. In the following we will show
\begin{equation}\label{e:pp}
\Delta_1\cap\sigma(A_0)\subset\spp(A_0).
\end{equation}
To this end set
$$
\wt\veps :=
\min\left\{\frac{\veps}{2},\,\frac{\delta^{j-2}\veps^j}{2^j(\|\Im N\| + r(\Im N))^{j-1}} : j = 2,\dots,k_0\right\},
$$
We may assume that $\Im N\neq 0$. Otherwise, the assertion of Theorem \ref{t:main} follows directly from Theorem \ref{t:lmm}. We will show that for all $\alpha\in\Delta_1\cap\sigma(A_0)$ and for all $x\in\calL_\Delta$ we have
$$
\|(A_0 - \alpha)x\|\le\wt\veps\|x\|\;\;\Longrightarrow\;\;[x,x]\ge\delta\|x\|^2,
$$
which then implies \eqref{e:pp}, see Lemma \ref{l:compact}. If $\sigma(\Im N)\cap\Delta = \emptyset$, then it follows from \eqref{e:basic} that $\calL_\Delta = \{0\}$, and nothing needs to be shown. Otherwise, there exists $\beta\in\Delta\cap\sigma(\Im N)$. Let $\alpha\in\Delta_1\cap\sigma(A_0)$ and $x\in\calL_\Delta$, $\|x\| = 1$, and suppose that $\|(A_0 - \alpha)x\|\le\wt\veps$. Let us prove that for all $j=1,\dots,2k_0$ we have
\begin{equation}\label{ImEst}
\left\|(B_0 - \beta)^jx\right\|\;\le\;\frac{\delta^{j-1}\veps^j}{2^j}.
\end{equation}
For $j = k_0,\dots,2k_0$ this is a direct consequence of \eqref{e:NormEst}. Assume now that \eqref{ImEst} holds for all $j\in\{k,\dots,k_0\}$ where $k\in\{2,\dots,k_0\}$ but does not hold for $j=k-1$, i.e.
\begin{equation}\label{ImDoof}
\left\|(B_0 - \beta)^{k-1}x\right\|\;>\;\frac{\delta^{k-2}\veps^{k-1}}{2^{k-1}}.
\end{equation}
Then we have
\begin{align*}
\bigg\|(N_0-(\alpha+i\beta))&\frac{(B_0 - \beta)^{k-1}x}{\|(B_0 - \beta)^{k-1}x\|}\bigg\|
\le \frac{\|B_0 - \beta\|^{k-1}\|(A_0 - \alpha)x\| + \|(B_0 - \beta)^{k}x\|}{\|(B_0 - \beta)^{k-1}x\|}\\
&\le \frac{2^{k-1}(\|\Im N\| + r(\Im N))^{k-1}}{\delta^{k-2}\veps^{k-1}}\wt\veps + \frac{2^{k-1}}{\delta^{k-2}\veps^{k-1}}\frac{\delta^{k-1}\veps^{k}}{2^{k}}\\
&\le \frac\veps 2 + \delta\frac\veps 2\le\veps.
\end{align*}
As $\alpha + i\beta\in\Delta_1\times\Delta\subset\calU$, it follows from \eqref{e:epsdel} that
$$
\delta\le\left[\frac{(B_0 - \beta)^{k-1}x}{\|(B_0 - \beta)^{k-1}x\|},\frac{(B_0 - \beta)^{k-1}x}{\|(B_0 - \beta)^{k-1}x\|}\right]\le\frac{\|(B_0 - \beta)^{2k-2}x\|}{\|(B_0 - \beta)^{k-1}x\|^2}.
$$
Owing to $k\le 2k - 2\le 2k_0$, relation \eqref{ImEst} holds for $j = 2k - 2$ by assumption, and thus
$$
\left\|(B_0 - \beta)^{k-1}x\right\|\le\sqrt{\delta^{-1}\|(B_0 - \beta)^{2k-2}x\|}\le\sqrt{\frac{\delta^{2k-4}\veps^{2k-2}}{2^{2k-2}}} = \frac{\delta^{k-2}\veps^{k-1}}{2^{k-1}}
$$
follows. But this contradicts \eqref{ImDoof}. Hence, \eqref{ImEst} holds for $j=k-1$, and, by induction, for $j=1$. Hence,
$$
\|(N - (\alpha+i\beta))x\|\le\|(A_0 - \alpha)x\| + \|(B_0 - \beta)x\|\le\veps.
$$
By \eqref{e:epsdel}, this yields $[x,x]\ge\delta$ and \eqref{e:pp} is proved.

Due to Theorem \ref{t:lmm} the operator $A_0\in L(\calL_\Delta)$ has a local spectral function $E_\Delta$ of positive type on $\Delta_1$, and the subspace
$$
\calH_{\Delta_1\times\Delta} := E_\Delta(\Delta_1)\calL_\Delta
$$
is the maximal spectral subspace of $A_0$ corresponding to $\Delta_1$. Moreover, $\calH_{\Delta_1\times\Delta}$ is a Hilbert space with respect to the inner product $\product$. Since $\calH_{\Delta_1\times\Delta}$ is invariant with respect to both $N$ and $N^+$, the $\product$-orthogonal complement $$\calH_{\Delta_1\times\Delta}^\gperp =
\left\{y\in \calH\, : \, [y,x]=0 \mbox{ for all } x\in
\calH_{\Delta_1\times\Delta} \right\}$$
is also $N$- and $N^+$-invariant and
$(\calH_{\Delta_1\times\Delta}^\gperp, \product)$ is a Krein space,
see e.g.\ \cite{L82}. Moreover, we have
$$
\left(N|\calH_{\Delta_1\times\Delta}\right)^+ = N^+|\calH_{\Delta_1\times\Delta}.
$$

{\bf 2. } Let $Q := \Delta_1\times\Delta\subset\calU$ be a rectangle as in step 1. By $Q^i$ ($\Delta^i$) we denote the complex (real, respectively) interior of the set $Q$ ($\Delta$, respectively). In this step of the proof we shall show that the subspaces $\calH_Q$ and $\calH_Q^\gperp$, defined in the first step, have the following properties.
\begin{enumerate}
\item[(a)] $\sigma\left(N|\calH_Q\right)\subset\sigma(N)\cap Q$.
\item[(b)] If $\calM\subset\calH$ is a subspace which is both $N$- and $N^+$-invariant such that
$$
\sigma(N|\calM)\subset Q,
$$
then $\calM\subset\calH_Q$.
\item[(c)] If $\calH_Q = \{0\}$ then $Q^i\subset\rho(N)$.
\item[(d)] $\sigma\left(N|\calH_Q^\gperp\right)\subset\ol{\sigma(N)\setminus Q}$.
\item[(e)] If the bounded operator $B$ commutes with $N$ then both $\calH_Q$ and  $\calH_Q^\gperp$ are $B$-invariant.
\item[(f)] $\calH_Q$ is the maximal spectral subspace of $N$ corresponding to $Q$.
\end{enumerate}
By Lemma \ref{l:invariant_spec} and \eqref{e:basic} we have
$$
\sigma(\Im (N|\calH_Q)) = \sigma(B_0|\calH_Q)\subset\sigma(B_0)\subset\Delta.
$$
In addition,
$$
\sigma(\Re(N|\calH_Q)) = \sigma(A_0|\calH_Q)\subset\Delta_1.
$$
From this and Lemma \ref{l:polynom} we obtain
$$
\sigma(N|\calH_Q)\subset Q.
$$
Since the spectrum of a normal operator in a Hilbert space coincides with its approximate point spectrum, (a) follows.

Let $\calM\subset\calH$ be a subspace as in (b). By Lemma \ref{l:polynom} we have
$$
\sigma(\Im N|\calM)\subset\Delta
\quad\;\text{ and }\;\quad
\sigma(\Re N|\calM)\subset\Delta_1.
$$
As $\calL_\Delta$ is the maximal spectral subspace of $\Im N$ corresponding to $\Delta$, we conclude from the first relation that $\calM\subset\calL_\Delta$. From the second relation we obtain (b) since $\calH_Q$ is the maximal spectral subspace of $\Re N|\calL_\Delta$ corresponding to the interval $\Delta_1$, cf.\
Theorem \ref{t:lmm}.

Let us prove (c). By definition of $\calH_Q$ it follows from $\calH_Q = \{0\}$ that $\Delta_1^i\subset\rho(\Re N|\calL_\Delta)$. Hence, by Lemma \ref{l:polynom} we have
\begin{equation}\label{e:fast_fast_c}
\Delta_1^i\times\R\;\subset\;\rho(N|\calL_\Delta).
\end{equation}
Let $J$ be a closed interval which contains $\sigma(\Im N)$ and let $\delta_1$ and $\delta_2$ be the two (closed) components of $J\setminus\Delta^i$. By $\calL_{\delta_1}$ and $\calL_{\delta_2}$ denote the maximal spectral subspaces of $\Im N$ corresponding to the intervals $\delta_1$ and $\delta_2$, respectively. Set
$$
\calL_{\Delta^c} := \calL_{\delta_1} \;\dotplus\; \calL_{\delta_2}.
$$
Obviously, we have
\begin{equation}\label{e:cc_sigma}
\sigma(\Im N|\calL_{\Delta^c})\,\subset\,\delta_1\cup\delta_2.
\end{equation}
And by \cite[Chapter II, Theorem 4]{LM} and \cite[Chapter I, \textsection 4.4]{LM} we have
\begin{equation}\label{e:cc_sum}
\calH = \calL_\Delta \;+\;\calL_{\Delta^c}.
\end{equation}
It is an immediate consequence of (b) that $\ker(N - \la)\subset\calH_Q = \{0\}$ for $\la\in Q^i$. Hence, due to \eqref{e:fast_fast_c} and \eqref{e:cc_sum}, it remains to show that $Q^i\subset\rho(N|\calL_{\Delta^c})$. But this follows directly from \eqref{e:cc_sigma} and Lemma \ref{l:polynom}.

Set $\wt N := N|\calH_Q^\gperp$. In order to show (d) we prove
\begin{equation}\label{e:anstatt}
\C\setminus\left(\ol{\sigma(N)\setminus Q}\right)\,\subset\,\rho(\wt N).
\end{equation}
Since
$$
\C\setminus\left(\ol{\sigma(N)\setminus Q}\right) = \rho(N)\cup Q^i\cup \{\la\in\partial Q : \nexists (\la_n)\subset\sigma(N)\setminus Q \mbox{ with } \lim_{n\to\infty}\,\la_n = \la\},
$$
and $\rho(N)\subset\rho(\wt N)$ by Lemma \ref{l:only_sap}, it suffices to show
\begin{equation}\label{uuzi}
 Q^i\cup \{\la\in\partial Q : \nexists (\la_n)\subset\sigma(N)\setminus Q
\mbox{ with }
 \lim_{n\to\infty}\,\la_n = \la\}\,\subset\,\rho(\wt N).
\end{equation}
Let $\la$ be a point contained in the set on the left hand side of this relation. Then there exists a compact rectangle $R = \Delta_1'\times\Delta'\subset\calU$ with $\la\in R^i$, $\ell(\Delta') < \tau$ and
$$
\sigma(N)\cap R\,\subset\, Q.
$$
Observe that the normal operator $\wt N$ in the Krein space $\calH_Q^\gperp$ satisfies the conditions of Theorem \ref{t:main}. In particular, relation \eqref{e:epsdel} holds with the same values $\veps$ and $\delta$ and with $N$ replaced by $\wt N$. Hence, there exists a subspace $\wt\calH_R$ of $\calH_Q^\gperp$ which is $N$- and $N^+$-invariant and has the properties
\begin{enumerate}
\item[($\wt a$)] $\sigma(\wt N|\wt\calH_R)\,\subset\,R\cap\sigma(\wt N)$,
\item[($\wt c$)] $\wt\calH_R = \{0\}\;\Longrightarrow\;R^i\subset\rho(\wt N)$.
\end{enumerate}
By virtue of (b) we conclude from ($\wt a$) and Lemma \ref{l:only_sap} that $\wt\calH_R\subset\calH_Q$. But since $\wt\calH_R$ is also a subspace of $\calH_Q^\gperp$, we have $\wt\calH_R = \{0\}$ which by ($\wt c$) implies $R^i\subset\rho(\wt N)$. Hence, $\la\in\rho(\wt N)$ and therefore \eqref{uuzi}
holds.

In order to prove (e) let $Q_n = \Delta_n'\times\Delta_n''\subset\calU$ be closed rectangles such that $\ell(\Delta_1'') < \tau$, $Q\subset Q_n^i$ for all $n\in\N$ and
$$
Q_1\supset Q_2\supset\dots\quad\;\text{ and }\;\quad Q = \bigcap_{n=1}^\infty\,Q_n.
$$
From (a) and (b) it follows that $\calH_Q\subset\bigcap_{n=1}^\infty\,\calH_{Q_n}$. Now, it is not difficult to see that $\C\setminus Q\subset\rho(N|\bigcap_{n=1}^\infty\,\calH_{Q_n})$, and (b) gives
\begin{equation}\label{e:subsp_cap}
\calH_Q = \bigcap_{n=1}^\infty\,\calH_{Q_n}.
\end{equation}
Let $E(Q)$ and $E(Q_n)$ be the $\product$-orthogonal projections onto the Hilbert spaces $\calH_Q$ and $\calH_{Q_n}$, respectively. As these spaces are invariant with respect to both $N$ and $N^+$, the projections commute with $N$. Let $B$ be a bounded operator which commutes with $N$ and let
 $B_Q\in L(\calH_Q,\calH)$ be the restriction of $B$ to $\calH_Q$. We obtain
\begin{align*}
\left(N|\calH_{Q_n}^\gperp\right)[(I - E(Q_n))B_Q]
&= (I - E(Q_n))NB_Q = [(I - E(Q_n))B_Q]\left(N|\calH_Q\right).
\end{align*}
The spectra of $N|\calH_{Q_n}^\gperp$ and $N|\calH_Q$ are disjoint by (a) and (d), and Rosenblum's Corollary (Theorem \ref{l:rosenblum}) implies $(I - E(Q_n))B_Q = 0$, i.e.\ $B\calH_Q\subset\calH_{Q_n}$ for every $n\in\N$. By \eqref{e:subsp_cap} this yields $B\calH_Q\subset\calH_Q$. Similarly, one shows that $B\calH_{Q_n}^\gperp\subset\calH_Q^\gperp$ for all $n\in\N$. From
$$
{\rm c.l.s.}\,\left\{\calH_{Q_n}^\gperp : n\in\N\right\}^\gperp = \bigcap_{n=1}^\infty\,\calH_{Q_n}
$$
and \eqref{e:subsp_cap} we deduce
$$
\calH_Q^\gperp = {\rm c.l.s.}\,\left\{\calH_{Q_n}^\gperp : n\in\N\right\}.
$$
Hence, for $x\in\calH_Q^\gperp$ there exists a sequence $(x_k)$ with each $x_k$ in some $\calH_{Q_{n_k}}^\gperp$ such that $x_k\to x$ as $k\to\infty$. Since $Bx_k\in\calH_Q^\gperp$ and $Bx_k\to Bx$ as $k\to\infty$, we conclude $Bx\in\calH_Q^\gperp$.

After all which has been proved above, for (f) we only have to show that every $N$-invariant subspace $\calM\subset\calH$ with $\sigma(N|\calM)\subset Q$ is a subspace of $\calH_Q$. Let $\calM$ be such a subspace. Then let $(Q_n)$ be a sequence of rectangles as in the proof of (e). From
$$
\left(N|\calH_{Q_n}^\gperp\right)\,[(I - E(Q_n))|\calM] = [(I - E(Q_n))|\calM]\,(N|\calM)
$$
and Rosenblum's Corollary we conclude $(I - E(Q_n))\calM = \{0\}$. Therefore, $\calM\subset\calH_{Q_n}$ for every $n\in\N$ and $\calM\subset\calH_Q$ follows from \eqref{e:subsp_cap}.

\
\\
\indent
{\bf 3. } In this step we complete the proof. Let $Q_1 = [a,b]\times\Delta_1
\subset \calU$ and $Q_2 = [a,b]\times\Delta_2\subset \calU$ such that $\ell(\Delta_j) < \tau$ for $j=1,2$ and assume that $\Delta_1$ and $\Delta_2$ have one common endpoint. Then $Q := Q_1\cup Q_2 = [a,b]\times (\Delta_1\cup\Delta_2)$ is also a closed rectangle. Define
$$
\calH_Q := \calH_{Q_1} + \calH_{Q_2} = \calH_{Q_1}[\dotplus]\left(\calH_{Q_1}^\gperp\cap\calH_{Q_2}\right).
$$
This is obviously a Hilbert space (with respect to $\product$) which is both $N$- and $N^+$-invariant. Let us prove that the statements (a)--(f) from part 2 of this proof also hold for $\calH_Q$. In step 2 the statements (d)--(f) were proved only with the help of (a)--(c). Here, this can be done similarly. Hence, it is sufficient to prove only (a)--(c). By (a$_j$)--(c$_j$) denote the corresponding properties of $\calH_{Q_j}$, $j=1,2$. Statement (a) holds since $N|\calH_Q$ is a normal operator in the Hilbert space $(\calH_Q,\product)$ and
$$
\sigma(N|\calH_Q) = \sigma(N|\calH_{Q_1})\cup\sigma(N|\calH_{Q_1}^\gperp\cap\calH_{Q_2})\subset Q_1\cup\sigma(N|\calH_{Q_2})\subset Q_1\cup Q_2.
$$
For (b) let $\calM$ be a  $N$- and $N^+$-invariant subspace with
$
\sigma(N|\calM)\subset Q
$. Denote by
 $\calL_{\Delta_j}^\calM\subset\calM$ be the maximal spectral subspace of $\Im
N|\calM$ corresponding to $\Delta_j$, $j=1,2$. Then, by Lemmas \ref{l:polynom}
and \ref{l:invariant_spec},
$$
\sigma(N|\calL_{\Delta_j}^\calM)\subset(\R\times\Delta_j)\cap(\sigma(N|\calM)\cup\rho_b(N|\calM))\subset(\R\times\Delta_j)\cap Q = Q_j.
$$
From (b$_j$) we obtain $\calL_{\Delta_j}^\calM\subset\calH_{Q_j}$, $j=1,2$. And since $\calM = \calL_{\Delta_1}^\calM + \calL_{\Delta_2}^\calM$
(see \cite[Chapter II, Theorem 4 and Chapter I, \textsection4.4]{LM})
 we have $\calM\subset\calH_Q$.

Suppose that $\calH_Q = \{0\}$. Then $\calH_{Q_1} = \calH_{Q_2} = \{0\}$ and hence $Q_1^i\cup Q_2^i\subset\rho(N)$ by (c$_1$) and (c$_2$). Let $R = [a,b]\times [c_1,c_2]$, where $c_j$ is the center of $\Delta_j$, $j=1,2$. Then $c_2 - c_1 < \tau$. From $\sigma(N|\calH_R)\subset R\subset Q$ and (b) it follows that $\calH_R\subset\calH_Q = \{0\}$. Hence, $R^i\subset\rho(N)$ which shows (c).

Now it is clear that for $Q = [a,b]\times [c,d]$ we choose a partition $c = t_0 < t_1 < \dots < t_m = d$ of $[c,d]$ such that $t_{k+1} - t_k < \tau$, $k=0,\dots,m-1$, and define
$$
\calH_Q := \calH_{Q_1} + \dots + \calH_{Q_m},
$$
where $Q_k := [a,b]\times [t_{k-1},t_k]$, $k=1,\dots,m$. This subspace is then a Hilbert space with respect to the indefinite inner product $\product$ with the properties (a)--(f). Moreover, $\calH_Q$ is both $N$- and $N^+$-invariant. Hence, $N|\calH_Q$ is a normal operator in the Hilbert space $(\calH_Q,\product)$ and has therefore a spectral measure $E_Q$. By $E(Q)$ we denote the $\product$-orthogonal projection onto $\calH_Q$. It is now easy to see that
$$
E(\cdot) := E_Q(\cdot)E(Q)
$$
satisfies conditions (i)-(iii) from Definition \ref{locspecfct}. The remaining
conditions (iv)-(vi) follow from (e), (a) and (d), respectively.
Hence, $E$
is the local spectral function of positive type of $N$ on $Q$.
\end{proof}

\begin{remark}
It is clear that under the conditions of Theorem \ref{t:main} the operator $N$ possesses a spectral function of positive type on open sets $S$ of positive type. In order to define $E(W)$ for $W\in\mathcal B(S)$, cover $W$ with finitely many closed rectangles $Q_1,\ldots,Q_n\in\mathcal B(S)$ and define the spectral projection $E(Q)$ for $Q = Q_1\cup\ldots\cup Q_n$ similarly as in the last part of the proof of Theorem \ref{t:main}. Then $N|E(Q)\calH$ is a normal operator in the Hilbert space $E(Q)\calH$ with spectrum in $Q$ and spectral measure $E_Q$, and $E(W)$ can be defined via $E_Q(W)E(Q)$.
\end{remark}

\begin{remark}
The statement of Theorem \ref{t:main} also holds if the growth condition on the imaginary part of $N$ is replaced by the (local) definitizability of $\Im N$ in the sense of P.\ Jonas (cf.\ \cite{J}) over a complex neighborhood of $[c,d]$.
\end{remark}

\section{Spectral sets of definite type}
In this section we show that Theorem \ref{t:main} also holds in the situation when the real part of $N$ is allowed to have nonreal spectrum but the set of definite type with respect to $N$ is a spectral set.

\begin{lemma}\label{Vladimir}
Let $N$ be a normal operator in the Krein space $(\calH,\product)$ and
let $\sigma$ be a spectral set of $N$ with
\begin{equation}\label{sapinspp}
\sigma\cap\sigma_{ap}(N) \subset \sigma_{++}(N).
\end{equation}
Then the Riesz-Dunford projection $Q$ of $N$ corresponding to $\sigma$ is selfadjoint in the Krein space $({\mathcal H},\product)$ and the corresponding spectral subspace $Q\calH$ is invariant with respect to both $N$ and $N^+$. Moreover, we have
$$
\sigma_{ap}(N|Q\calH) \subset \sigma_{++}(N|Q\calH).
$$
\end{lemma}
\begin{proof} Since $N$ is normal, $Q$ is also normal,
hence it commutes with $Q^+$. Moreover, $Q^+$ is the Riesz-Dunford
projection corresponding to $N^+$  and the set $\{\lambda :
\overline{\lambda} \in \mathcal U\}$, so $Q^+$ also commutes with
$N$. Thus, the projection $Q-Q^+Q$ projects on a subspace $\calM$
which is invariant with respect to $N$. This subspace is neutral.
Hence, from \eqref{sapinspp} it follows that $\sigma_{ap}(N|\calM) =
\emptyset$. This is only possible if $\calM = \{0\}$, and we
conclude
$$
Q=Q^+Q,
$$
that is, $Q$ is a selfadjoint projection. The last statement follows
from $\sigma_{ap}(N|Q\calH) = \sigma\cap\sigma_{ap}(N)$.
\end{proof}

\begin{theorem}\label{t:sset_main}
Let $N$ be a normal operator in the Krein space $(\calH,\product)$.
Let $\sigma$ be a spectral set of $N$ with
$$
\sigma\cap\sigma_{ap}(N)\subset\sigma_{++}(N),
$$
and let $Q$ be the Riesz-Dunford projection corresponding to
$\sigma$ and $N$. Assume that
$$
\sigma(\Im N|Q\calH)\subset\R\quad \left(\text{or }\,\sigma(\Re
N|Q\calH)\subset\R\right)
$$
and that the growth of the resolvent of $\Im N|Q\calH$ {\rm (}$\Re
N|Q\calH$, respectively{\rm )} is of finite order. Then the spectral
subspace $Q\calH$ equipped with the inner product $\product$ is a
Hilbert space. Hence, the restriction $N|Q\calH$ is a normal
operator in the Hilbert space $(Q\mathcal H,\product)$ and,
therefore, possesses a spectral function.
\end{theorem}
\begin{proof}
By Lemma \ref{Vladimir} the space $(Q\calH,\product)$ is a
Krein space and $Q\calH$ is $N^+$-invariant.
Hence $(N|Q\calH)^+ = N^+|Q\calH$ and $N|Q\calH$ is
  normal in $Q\calH$. 
Therefore it is no restriction to assume $\sigma_{ap}(N) =
\sigma_{++}(N)$, $\sigma(\Im N)\subset\R$ and that the resolvent of
$\Im N$ is of finite order $k_0$ for some $k_0\in\N$. For each
compact interval $\Delta$ denote the maximal spectral subspace
corresponding to $\Im N$ and $\Delta$ (which exists due to
\cite{LM}) by $\calL_\Delta$.

It is a consequence of Lemma \ref{l:compact} that there exist
$\veps,\delta > 0$ with $\delta<1$ such that for all $\mu\in K$,
$$
K := \{\la + ib : \la\in\sigma(\Re N),\,b\in\sigma(\Im N)\},
$$
and all $x\in\calH$ we have
\begin{equation}\label{epsdelta}
\|(N - \mu)x\|\le\veps\|x\|\quad\Longrightarrow\quad
[x,x]\ge\delta\|x\|^2.
\end{equation}

Let $b\in\sigma(\Im N)$. From Corollary \ref{c:power} it follows that
there exists a compact interval $\Delta$ with center $b$ such that
\begin{equation}\label{NormEst}
\left\|(\Im N|\calL_\Delta -
b)^k\right\|\,\le\,\frac{\delta^{k-1}\veps^k}{2^k} \quad\text{for
all }\,k = k_0,k_0+1,\dots,2k_0,
\end{equation}
where $k_0$ is the order of growth of the resolvent of $\Im N$.

Since the subspace $\calL_\Delta$ is hyperinvariant with respect to $\Im N$, it is $\Re N$-invariant. The operator $\Re N|\calL_\Delta$ is a bounded operator in $\calL_\Delta$ which is  $\product$-selfadjoint in the sense that
$$
[(\Re N) x,y] = [x,(\Re N) y]\quad\text{for all}\,x,y\in\calL_\Delta,
$$
cf.\ \eqref{Gsa}.
We define
$$
\wt\veps :=
\min\left\{\frac{\veps}{2},\,\frac{\delta^{j-2}\veps^j}{2^j(\|\Im
N\| + r(\Im N))^{j-1}} : j = 2,\dots,k_0\right\}.
$$
In a similar way as  in step 1 of the proof of Theorem \ref{t:main}  
it is shown here that from $\|(\Re N -\la)x\|\le\wt\veps\|x\|$ for $x\in\calL_\Delta$, $\|x\|=1$ and $\la\in\sigma(\Re N|\calL_\Delta)$ it follows that $\|(\Im N - b)x\|\le\frac\veps 2 \|x\|$ and thus
$$
\|(N - (\la+ib))x\| \le \|(\Re N - \la)x\| + \|(\Im N - b)x\| \le \veps.
$$
Thus, with \eqref{epsdelta}, we obtain
$$
\sap(\Re N|\calL_\Delta)\;\subset\;\spp(\Re N|\calL_\Delta).
$$
Since $\spp(\Re N|\calL_\Delta)\subset\R$ (see Remark \ref{r:spp_real}) we conclude that $\C\setminus\R\subset
\C \setminus \sap(\Re N|\calL_\Delta)$. But as $\Re N|\calL_\Delta$ is bounded we even have $\C\setminus\R\subset\rho(\Re N|\calL_\Delta)$ and thus
$$
\sigma(\Re N|\calL_\Delta) = \sap(\Re N|\calL_\Delta)=\spp(\Re N|\calL_\Delta).
$$
It is now a consequence of \cite[Theorem 3.1]{LMM} that $(\calL_\Delta,\product)$ is a Hilbert space. It is easily seen that also the subspace $\calL_\Delta^\gperp$ is invariant with respect to $\Im N$. Consider the operator $A := \Im N|\calL_\Delta^\gperp$. If $\Delta_1$ is a compact interval which is completely contained in the inner of $\Delta$, then by \cite{LM} there exists a spectral subspace $\calL_{\Delta_1}\subset\calL_\Delta^\gperp$ of $A$ such that $\sigma(A|\calL_{\Delta_1})\subset\Delta_1$. But as this implies $\sigma(\Im N|\calL_{\Delta_1})\subset\Delta$ and $\calL_\Delta$ is a {\it maximal} spectral subspace, we obtain $\calL_{\Delta_1}\subset\calL_\Delta$ and thus $\calL_{\Delta_1}\subset\calL_\Delta\cap\calL_\Delta^\gperp = \{0\}$. Hence, $b\in\rho(\Im N|\calL_\Delta^\gperp)$ follows.

We are now ready to prove $b\in\sigma_{++}(\Im N)$. Let
$(x_n)\subset\calH$ be a sequence with $\|x_n\| = 1$, $n\in\N$, and
$(\Im N - b)x_n\to 0$ as $n\to\infty$. Write
$$
x_n = u_n + v_n\quad\text{ with }\quad
u_n\in\calL_\Delta,\,v_n\in\calL_\Delta^\gperp.
$$
From $(\Im N - b)x_n\to 0$ it follows that also $(\Im N - b)v_n\to 0$,
and $b\in\rho(\Im N | \calL_\Delta^\gperp)$ implies $v_n\to 0$ as
$n\to\infty$. From the fact that $(\calL_\Delta,\product)$ is a Hilbert
space we conclude
$$
\limsup_{n\to\infty}\,[x_n,x_n] = \limsup_{n\to\infty} \, \left(
[u_n,u_n] + [v_n,v_n]\right) = \limsup_{n\to\infty}\,[u_n,u_n] > 0.
$$
Since $b\in\sigma(\Im N)$ was arbitrary, we have
$\sigma(\Im N) = \sigma_{++}(\Im N)$, and it follows from, e.g.,
\cite[Theorem 3.1]{LMM} that $(\calH,\product)$ is a Hilbert space.
\end{proof}

In \cite{B98} a bounded normal operator $N$ in a Krein space is called {\it strongly stable} if there exists a fundamental decomposition \eqref{funddecomp}
such that $\calH_+$ and $\calH_-$ are invariant subspaces with respect to $N$
with $\sigma(N|\calH_+)\cap\sigma(N|\calH_-)=\emptyset$.
The following Theorem  \ref{t:Nuu} provides a new characterization of strongly stable normal operators
in Krein spaces. We say that an operator $T\in L(\calH)$ is similar to a selfadjoint (normal) operator in a Hilbert space if there exists a Hilbert space scalar product $\hproduct$ on $\calH$ which induces the topology of $(\calH,\product)$ such that $N$ is selfadjoint (normal, respectively) in the Hilbert space $(\calH,\hproduct)$.

\begin{theorem}\label{t:Nuu}
A normal operator $N$ in the Krein space $(\calH,\product)$ is strongly stable if and only if
\begin{equation}\label{Coni}
\sigma(N) = \sigma_{++}(N) \cup \sigma_{--}(N),
\end{equation}
\begin{equation}\label{Coni2}
\sigma(\Im N)\subset\R\quad \left(\text{or }\,\sigma(\Re
N)\subset\R\right)
\end{equation}
and  
\begin{equation}\label{add-str}\mbox{the growth of the resolvent of }\Im N \ {\rm (}\Re N\mbox{,
respectively}{\rm )}\ \mbox{ is of finite order.}\end{equation} In particular, in this case,  $N$ is similar to a
normal operator in a Hilbert space.
\end{theorem}
\begin{proof}
Let $N$  be strongly stable. Then \eqref{Coni} follows and \eqref{Coni2}
and the growth condition follow from the fact that $\Im N|\calH_\pm$
and $\Re N|\calH_\pm$ are selfadjoint operators in the Hilbert spaces
$(\calH_+,\product)$ and $(\calH_-,-\product)$, respectively.

For the converse observe that the sets $\sigma_{++}(N)$ and $\sigma_{--}(N)$ are open in $\sigma(N)$, see Lemma \ref{l:compact}. Therefore, $\sigma_{++}(N)$ and $\sigma_{--}(N)$ are spectral sets. Let $Q_+$ and $Q_-$ be the spectral projections corresponding to these sets, respectively. Then, since $Q_+Q_- = 0$, due to Theorem \ref{t:sset_main} the operator $J := Q_+ - Q_-$ is a fundamental symmetry in $(\calH,\product)$ with the desired properties.

In order to show the last statement of Theorem \ref{t:Nuu}, we denote by $N^*$ the adjoint of $N$ with respect to the Hilbert space inner product $[J\cdot,\cdot]$. Then, from $N^* = N^+$ it follows that $N$ is a normal operator in $(\calH,[J\cdot,\cdot])$.
\end{proof}

The following theorem shows that \eqref{Coni2} and \eqref{add-str}
in Theorem \ref{t:Nuu} can be replaced by the condition that $\Re N$ and $\Im N$ have real spectra and that $N$ is similar to a normal operator in a Hilbert space.

\begin{theorem}\label{sim-norm}
Assume that the normal operator $N$ in the Krein space $(\calH,\product)$ is similar to a normal operator in a Hilbert space and that $\sigma(\Re N)\subset\R$ and $\sigma(\Im N)\subset\R$. Then $\Re N$ and $\Im N$ are similar to selfadjoint operators in a Hilbert space. In particular, their resolvent growths are of first order.
\end{theorem}
\begin{proof}
Let $\mathcal B$ denote the set of all Borel-measurable subsets of $\mathbb C$ and set $Q^* := \{\ol\la : \la\in Q\}$ for $Q\in\mathcal B$. Moreover, let $\hproduct$ be a Hilbert space scalar product on $\calH$ with respect to which $N$ is normal, let $G\in L(\calH)$ such that $\product = (G\cdot,\cdot)$ and let $E$ be the spectral measure of the normal operator $N$ in $(\calH, (G\cdot,\cdot))$. Then $E_*$, defined by $E_*(Q) := E(Q^*)$, $Q\in\mathcal B$, is the spectral measure of $N^*$. It follows from the properties of $E_*$ that the function $E_+$ given by $E_+(Q) = G^{-1}E_*(Q)G = E_*(Q)^+ = E(Q^*)^+$, $Q\in\mathcal B$, is a countably additive resolution of the identity for $N^+ = G^{-1}N^*G$ (where we here now use the notions of \cite[Section XV.2]{DS}), that is, $N^+$ is a spectral operator in the sense of Dunford \cite{DS}.

Note that for any compact rectangle $Q\subset\C$ of the type $Q=[a,b]\times [c,d]$ the projection $E(Q)$ (and therefore the projection $E_*(Q)$) commutes with any operator that commutes with $N$, so the operators $\Re (N|{E_*(Q)\calH})$ and $\Im (N|{E_*(Q)\calH})$ have real spectra (cf.\ Lemma \ref{l:invariant_spec}). From Lemma \ref{l:polynom} we conclude that $\sigma(N^+|E_*(Q)\calH)\,\subset\,Q$. Since $E_+(Q)\calH$ is the maximal spectral subspace of $N^+$ corresponding to $Q$ (see, e.g., \cite[Example 2.1.6(ii)]{CF}), we have $E_*(Q)\calH\subset E_+(Q)\calH$ and hence
$$
E(Q^*)^+E(Q^*) = E_+(Q)E_*(Q) = E_*(Q) = E(Q^*).
$$
Therefore, for all compact rectangles $Q\subset\C$ the projection $E(Q)$ is selfadjoint in the Krein space $(\calH,\product)$. Since the system of compact rectangles in $\C$ is stable with respect to intersections and generates $\mathcal B$, it follows that $E(Q) = E(Q)^+$ for all $Q\in\mathcal B$.  This implies that for all $Q\in\mathcal B$ we have $GE(Q) = GE(Q)^+ = GG^{-1}E(Q)G = E(Q)G$ and thus $GN = NG$. Consequently, $N^+ = G^{-1}N^*G = N^*$, which implies the assertion.
\end{proof}

\begin{remark}
Theorem \ref{sim-norm} shows that the  growth of the resolvent of $\Im N $ and $\Re N $ under the conditions of Theorem \ref{t:Nuu} are of first order. 
However in the formulation of Theorem \ref{t:Nuu} we choose the
(formally) weaker Condition \eqref{add-str}.
\end{remark}

\section{Quadratic operator pencils with normal coefficients}
In this section we apply our results to operator pencils. A standard
description of damped small oscillations of a continuum or of  small oscillations of a
pipe, carrying  steady-state fluid of ideal incompressible fluid, is done via
an equation of the form
\begin{equation}\label{erstes}
T \ddot{z} +  R \dot{z} +Vz =0,
\end{equation}
where $z$ is a function with values in a Hilbert space and $V$ and $R$
are (in general) unbounded operators. As a reference (especially for
 non-selfadjoint coefficients) we mention here only
\cite{shk94,JT09} and \cite[Chapters 6.4 and 6.5]{KK01}.

The classical approach (see \cite{KL1,KL2}) to such kind of problems is,
under some additional assumptions  ($V$ uniformly positive and
the closures of the operators $V^{-1/2}TV^{-1/2}$ and $V^{-1/2}RV^{-1/2}$
are bounded) to transform the equation in \eqref{erstes}
via $u=V^{1/2}z$ into
\begin{equation}\label{zweites}
 E\ddot{u} +  F \dot{u} +u =0,
\end{equation}
with bounded operators $E$ and $F$.
If one is interested in finding solutions of the form
\begin{equation*}
u(t) = e^{t\lambda^{-1}} \phi_0,
\end{equation*}
with a constant vector $\phi_0$, then \eqref{zweites} can be written
(after multiplication by $\lambda^2$) as
\begin{equation}\label{drittes}
 (\lambda^2 I + \lambda E + F)  \phi_0 =0.
\end{equation}
In the sequel, we will investigate quadratic pencils of the form \eqref{drittes}
with $E = AC$ and $F = C^2$, where
\begin{equation}\label{Assumptions1}
C\mbox{ is a bounded normal operator in a Hilbert space } H
\end{equation}
and
\begin{equation}\label{Assumptions2}
A\mbox{ is a bounded selfadjoint operator in } H \mbox{ which commutes with } C.
\end{equation}
That is, we investigate the operator pencil $L$,
\begin{equation}\label{viertes}
L(\lambda) := \lambda^2 I + \lambda AC + C^2.
\end{equation}
As usual, a value $\lambda$ for which the equation $L(\lambda)\phi=0$
has a solution $\phi\ne 0$ is called an {\it eigenvalue of the operator pencil $L$}
and the {\it spectrum $\sigma(L)$ of $L$} is the set of all complex numbers
$\lambda$ for which the operator $L(\lambda)$ is not boundedly invertible.
In many cases it turns out (see, e.g., \cite{KL1,KL2}) that a successful
investigation of the spectral properties of $L$ is achieved by studying
the operator roots $Z$ of the quadratic operator equation
\begin{equation}\label{quadraticoperatorequation}
Z^2 + ACZ + C^2=0.
\end{equation}
If there exists a bounded operator $Z_1$ which is an operator root,
i.e., a solution of \eqref{quadraticoperatorequation}, then any eigenvalue
(eigenvector) of $Z_1$ is also an eigenvalue
(eigenvector, respectively)  of the operator pencil $L$. Moreover
$\partial \sigma(Z_1) \subset \sigma(L)$ (see \cite[Lemma 22.10]{Markus})
and the operator pencils $L$ decomposes into linear factors
\begin{equation*}\label{quadraticdecomposes}
L(\lambda) = (\lambda I - \widehat Z_1)  (\lambda I -  Z_1),
\end{equation*}
where  $\widehat Z_1 = -AC-Z_1$.

The following theorem on the existence of an operator root of \eqref{quadraticoperatorequation} shows how our previous results
can be applied.
\begin{theorem}
Assume that the coefficients $A$ and $C$
 of the operator pencil $L(\lambda)$ in \eqref{viertes}
satisfy \eqref{Assumptions1} and \eqref{Assumptions2}.
Define on the Hilbert space $\mathcal H:= H\times H$ an inner product by
\begin{equation} \label{Krein}
\left[ \left(\begin{smallmatrix} x_{1}\\
y_{1}\end{smallmatrix}\right),\left(\begin{smallmatrix} x_{2}\\
y_{2}\end{smallmatrix}\right)\right] := ( x_{1},x_{2})-
( y_{1},y_{2} ) \quad \mbox{for }\left(\begin{smallmatrix} x_{1}\\
y_{1}\end{smallmatrix}\right),\left(\begin{smallmatrix} x_{2}\\
y_{2}\end{smallmatrix}\right) \in  H\times H,
\end{equation}
where $\Skdef$ denotes the Hilbert space scalar product in $H$.
Then the operator matrix $\mathcal A$
\begin{equation*}\label{begleit}
 \mathcal A = \left[\begin{array}{cc} 0 & C\\ -C & -AC
\end{array}\right].
\end{equation*}
is a normal operator in the Krein space $(H\times H, \Skindef )$.
If the operator $\calA$ satisfies the conditions in Theorem \ref{t:Nuu},
then equation \eqref{quadraticoperatorequation} has an operator root.
\end{theorem}
\begin{proof}
Obviously, $\mathcal H= H\times H$ with inner product \eqref{Krein}
is a Krein space and the adjoint of $\mathcal A$ with respect to
 $\product$ is given by
\begin{equation*}\label{defj}
 \mathcal A^+= \left[\begin{array}{cc} 0 & C^*\\ -C^* & -AC^*
\end{array}\right].
\end{equation*}
From this, we easily conclude $\mathcal A \mathcal A^+=\mathcal A^+ \mathcal A$.
If the operator $\calA$ satisfies the conditions in Theorem \ref{t:Nuu},
 then $\calA$ is a strongly stable normal operator in
the Krein space $(H \times H, \product )$.
Hence,  there exists a fundamental decomposition \eqref{funddecomp}
such that $\calH_+$ and $\calH_-$ are invariant subspaces with respect to $\mathcal A$. Let  $K:H\to H$ with
$\|K\| <1$ be the corresponding angular operator, see, e.g., \cite[Chapter 1, \S8]{AI}, such that
\begin{equation*}\label{AngularOp}
\calH_+ =\left\{\left(\begin{array}{c}
x_+ \\ Kx_+
\end{array}  \right) : x_+ \in H\right\}.
\end{equation*}
 Now, $\mathcal A \calH_+\subset\calH_+$
implies that for every $x_+ \in H$ there exists $y_+ \in H$ with
\begin{equation*}\label{AngularOp2}
\left(\begin{array}{c}
CKx_+ \\ -Cx_+  -ACKx_+
\end{array}  \right) = \left(\begin{array}{c}
y_+ \\ Ky_+
\end{array}  \right)
\end{equation*}
and we obtain
$$
-C-ACK= KCK.
$$
Multiplication by $C$ from the right gives
$$
(KC)^2+ AC(KC) + C^2 =0,
$$
and the operator $KC$ is an operator root of  \eqref{quadraticoperatorequation}.
\end{proof}

\begin{remark}
If the operator pencil $\hat L(\la) := \la^2 + \la AD + D^2$ with $D = \frac 1 2(C+C^*)$ is hyperbolic, i.e.
\begin{equation}\label{e:hyp}
(ADx,x)^2\,\ge\,4(D^2x,x)\quad\text{for all }x\in H,\,\|x\|=1,
\end{equation}
and if $0$ is not an eigenvalue of $D$,
then equation \eqref{quadraticoperatorequation} has an operator root $Z_1$ which commutes with $\hat Z_1$. To see this, we note that \eqref{e:hyp} implies $(A^2x,x)\|Dx\|^2 = \|Ax\|^2\|Dx\|^2\ge(Ax,Dx)^2\ge 4\|Dx\|^2$ and hence $A^2 - 4\ge 0$. Let $W := (A^2 - 4)^{1/2}$. Now, a simple computation shows that both
$$
Z_1 := \frac1 2(W-A)C\quad\text{ and }\quad Z_2 := -\frac 1 2(W+A)C
$$
are operator roots of \eqref{quadraticoperatorequation}, $Z_1Z_2 = Z_2Z_1$ and $L(\la) = (\la - Z_1)(\la - Z_2)$.
\end{remark}


\section*{Acknowledgements}
Vladimir Strauss and Friedrich Philipp gratefully acknowledge support by DFG, Grants TR 903/9-1 and  TR 903/4-1, respectively.

\section*{Contact information}
Friedrich Philipp: Institut f\"ur Mathematik, MA 8-1, Technische Universit\"at Berlin, Stra\ss e des 17.\ Juni 136, 10623 Berlin, Germany, fmphilipp@gmail.com

\vspace{0.4cm}\noindent
Vladimir Strauss: Universidad Sim\'on Bol\'ivar, Departamento de Matem\'aticas Puras y Aplicadas, Apartado 89.000, Caracas 1080-A, Venezuela, str@usb.ve

\vspace{0.4cm}\noindent
Carsten Trunk: Institut f\"ur Mathematik, Technische Universit\"at Ilmenau, Postfach 10 05 65, 98684 Ilmenau, Germany, carsten.trunk@tu-ilmenau.de
\end{document}